\documentclass{amsart}

\usepackage{amsmath}
\usepackage{amsthm}
\usepackage{amsfonts}
\usepackage{amssymb}

\newcommand\F{{\bf F}}

\newcommand\I{{\mathbf{I}}}
\newcommand\E{{\mathbf{E}}}
\renewcommand\P{{\mathbf{P}}}
\newcommand\eps{\varepsilon}

\newcommand\sgn{\operatorname{sgn}}

\newcommand\Per{\operatorname{Per}}
\newcommand\Det{\operatorname{Det}}

\theoremstyle{plain}
  \newtheorem{theorem}[subsection]{Theorem}

  \newtheorem{proposition}[subsection]{Proposition}
  
  \newtheorem{lemma}[subsection]{Lemma}
  \newtheorem{corollary}[subsection]{Corollary}

\theoremstyle{remark}
  \newtheorem{remark}[subsection]{Remark}

\theoremstyle{definition}

\include{psfig}

\begin{document}

\title{On the permanent of random Bernoulli matrices}

\author{Terence Tao}
\address{Department of Mathematics, UCLA, Los Angeles CA 90095-1555}
\email{tao@@math.ucla.edu}
\thanks{T. Tao is supported by NSF grant CCF-0649473 and a grant from the MacArthur Foundation.}

\author{Van Vu}
\address{Department of Mathematics, Rutgers University, Piscataway NJ 08854-8019}
\email{vanvu@@math.rutgers.edu}
\thanks{V. Vu  is supported by an NSF Career Grant.}


\begin{abstract}
We show that the permanent of an $n \times n$ matrix with iid
Bernoulli entries $\pm 1$ is of magnitude $n^{(\frac{1}{2}+o(1))n}$
with probability $1-o(1)$. In particular, it is  almost surely
non-zero.
\end{abstract}

\maketitle

\section{Introduction}

Let $M$ be an $n \times n$ matrix. Two basic parameters of $M$ are
its \emph{determinant}

$$ \Det(M) := \sum_{\sigma \in S_n} \sgn(\sigma) \prod_{i=1}^n a_{i\sigma(i)},$$

and its \emph{permanent}

$$ \Per(M) := \sum_{\sigma \in S_n} \prod_{i=1}^n  a_{i\sigma(i)}.$$

Let $M_n$ denote the random Bernoulli matrix of size $n$ (the
entries of $M_n$ are iid random variables taking values $\pm 1$ with
probability $1/2$ each).  For some time, it has been a central
problem in probabilistic combinatorics to determine the asymptotic
behavior of $\Det (M_n)$ and $\Per (M_n)$, as $n$ tends to infinity
(here and later we use the asymptotic notation under this
assumption).

In the 1960s,  Koml\'os  \cite{Kom1, Kom2} proved that asymptotically almost
surely (i.e. with probability $1-o(1)$), $\Det M_{n} \neq 0$. Since then, the problem of estimating
the singular probability $\P(\Det M_{n} =0)$ was studied in many
papers \cite{ Kom3, KKS,TV2, BVW}.  It is easy to see that  $  \P(
\Det(M_n) =0) \ge (1/2+o(1))^{n}$
 and it  has been conjectured that this lower bound is sharp. The most current upper bound is
 $\P( \Det(M_n) = 0 ) \leq (\frac{1}{\sqrt{2}}+o(1))^n$ \cite{BVW}.

The order of magnitude of $\Det M_{n}$ was computed recently.  In
\cite{TV1}, the authors showed that

\begin{equation}\label{detm} \hbox {\it Asymptotically almost surely,}\,\,
|\Det(M_n)| = n^{(1/2-o(1))n}.
\end{equation}

On the other hand, little has been known about the permanent. Prior
to this work, it was not known whether $\Per (M_{n}) \neq 0$ almost
surely. It was  observed by Alon (see also \cite{Wan}) that if $n+1$
is a power of $2$, then any $n\times n$ $\pm 1$ matrix has permanent
equal $(n+1)/2$ modulo $n+1$ and thus is non-zero.

Similar to the situation with $\Det$, the second moment of $\Per$ is
easy to compute, using the definition of permanent and  linearity of
expectation

\begin{equation}\label{variance}
\E |\Per(M_n)|^2 = n!.
\end{equation}

Few higher moments of $\Per(M_n)$ can also be computed (with some
difficulty), but they do not appear to reveal much useful
information.

\vskip2mm

The main goal of this paper is to establish an analogue of
\eqref{detm} for $\Per (M_{n})$.

\begin{theorem} \label{theorem:main} Asymptotically almost surely,

$$|\Per(M_n)| = n^{(\frac{1}{2}+o(1))n}. $$ \end{theorem}

The upper bound follows from \eqref{variance}, Chebyshev's
inequality and the fact that $n!= n^{(1+o(1))n}$. The main task is
to prove the lower bound and we are going to show

\begin{theorem}\label{main-thm}  There is a positive constant $c$ such that for every $\eps > 0$
and $n$ sufficiently large depending on $\eps$, we have

$$ \P( |\Per (M_{n})| \ge n^{(\frac{1}{2}-\eps)n}) \ge  1-n^{-c} .$$

\end{theorem}

\begin{remark} The constant $c > 0$ in Theorem \ref{main-thm} can be made explicit
(e.g. one can take $c=1/10$) but we have not attempted to optimise
it here. In any case, our method does not seem to lead to  any value
of $c$ larger than $1/2$, due to its reliance on the
Erd\H{o}s-Littlewood-Offord inequality (Lemma \ref{lo}) at the very
last step (to get from $(n-1) \times (n-1)$-minors to the $n \times
n$ matrix).  In principle, one can obtain better results by using
more advanced Littlewood-Offord inequalities, but it is not clear to
the authors how to restructure the rest of the argument so that such
inequalities can be exploited.
\end{remark}

\begin{remark} The lower bound $n^{(\frac{1}{2}-\eps)n} $ can probably be sharpened to
$n^{n/2} \exp(-\Theta (n))$, but we do not pursue this direction here.
\end{remark}

\begin{remark} Our proof also works (verbatim) for $\Det$ and thus we obtains a
 new proof for
\eqref{detm}. The lower bound obtained for the determinant  is
however inferior to that in \cite{TV1}.
\end{remark}

\begin{remark} The Bernoulli distribution does not play a
significant role. The theorem holds for virtually any (not too
degenerate) discrete distribution. Also, it is not necessary to
assume that the entries have identical distribution. The
independence is, however, critical.  In particular, our arguments do not seem to easily yield any non-trivial result for the permanent of a random \emph{symmetric} Bernoulli matrix.
\end{remark}

All previous proofs concerning
 $\Det(M_n)$ proceeded by geometric arguments
 (for instance, interpreting $\Det(M_n)=0$ as the event that the rows of $M_n$ lie in a hyperplane).
  Such geometric arguments are unavailable for the permanent and thus one needs to find a new approach.
  In this paper, we proceed   by a  combinatorial method, studying the propagation of probabilistic
   lower bounds for the permanent from small minors to large ones.
   Roughly speaking, we are going to expose the rows of the matrix one at the time and
   try to show that, with high probability, the magnitude of the permanent of many (full-size) minors
   increases by a large factor (close to $\sqrt n$) at every
   step. This can be done in most of the process except the last few
   steps, where we simply  keep the permanents  from dropping.

\vskip2mm

  In the next section, we present our probabilistic tools. The proof
  is outlined in Section \ref{section:mainproof}, modulo many
  propositions. The rest of the paper is devoted to the verification
  of these propositions. As already mentioned, we are going to use the standard asymptotic
  notation ($O,o, \Omega, \Theta$) under the assumption that $n
  \rightarrow \infty$.

\begin{remark}  Random matrices in which the entries are bounded away from zero were studied in \cite{rempala}, \cite{rempala2}.  In this situation there is much less cancellation and a stronger result is known, namely a central limit theorem for the permanent.  For random 0-1 matrices, the problem is closely related to that of counting perfect matchings in a random graph \cite{janson}. We also mention that some general results for the permanent rank of a matrix $A$ (i.e. the size of the largest minor of $A$ with non-vanishing permanent) were established in \cite{Yu}.
\end{remark}

\section{Probabilistic tools}

We shall rely frequently on three standard tools from probability theory.  The first one asserts that if there are a collection of events that are individually likely to be true, then it is likely that most of them are true at once, even if there are strong correlations between such events:

\begin{lemma}[First moment]\label{first}  Let $E_1,\ldots,E_m$ be arbitrary events
(not necessarily independent) such that $\P(E_i) \geq 1-\delta$ for all $1 \leq i \leq m$ and some $\delta > 0$, and let $0 < c < 1$.  Then
$$ \P( \hbox{At most } cm \hbox{ of the } E_1,\ldots,E_m \hbox{ are false} ) \geq 1 - \frac{\delta}{c}.$$
\end{lemma}

\begin{proof}  Let $I(E)$ be the indicator of an event $E$.  From Markov's inequality we have
$$ \P( \sum_{i=1}^m I(\overline{E_i}) \geq cm ) \leq \frac{1}{cm} \E \sum_{i=1}^m I(\overline{E_i}),$$
and the claim follows from linearity of expectation.
\end{proof}

Our next tool is the following concentration result, a well known
consequence of Azuma's inequality \cite{AS}.

\begin{lemma} \label{Azuma0} Let $T>0$, let $\xi_{1}, \dots, \xi_{n}$ be iid
Bernoulli variables, and let $Y=Y(\xi_{1}, \dots, \xi_{n})$ be a function such that
$|Y(x)- Y(x')| \le T$ for all pairs $x=(\xi_{1}, \dots, \xi_{n}), x'=(\xi_{1}', \dots, \xi_{n}')$
 of Hamming distance one. Then

$$\P (|Y-\E(Y))| \ge S) \le 2 \exp(-\frac{S^{2}}{2nT^{2}}). $$ \end{lemma}

We also need the following (also standard) one-sided version of
Azuma's inequality, which can be proved in the same way as Azuma's
inequality itself.

\begin{lemma}\label{azuma}  Let $\F_0 \subset \F_2 \dots \subset \F_m$ be
a sequence of nested $\sigma$-algebras in a probability space
$\Omega$ and $W_i$, $0\le i \le m$, be $\F_i$-measurable real
functions obeying the submartingale-type property
$$ \E( W_i | \F_{i-1} ) \leq W_{i-1}$$
for all $1 \leq i \leq m$.  Assume also that $|W_i - W_{i-1}| \leq 1$
for all $1 \leq i \leq m$.  Then for any $\lambda \geq 0$ we have
$$ \P( W_m - W_0 \geq \lambda ) \leq \exp( -\frac{\lambda^2 }{ 2m} ).$$
\end{lemma}


Finally, we need the classical Littlewood-Offord-Erd\H{o}s
inequality \cite{Erd1}.

\begin{lemma}
\label{lo}  Let $\lambda > 0$ and $m, k \geq 1$, and let
$v_1,\ldots,v_m$ be real numbers such that $|v_i| \geq \lambda$ for
at least $k$ values of $i$.  Let $a_1,\ldots,a_m$ be iid signs drawn
uniformly from $\{-1,+1\}$.  Then we have
$$ \P( |a_1 v_1 + \ldots + a_m v_m| \leq x \lambda ) =O( \frac{x}{\sqrt{k}})$$
for all $x \geq 1$.
\end{lemma}

\section{Preliminary reductions} \label{section:mainproof}

Fix a small $\eps_0 >0$.  Our goal is
 to show that

\begin{equation} \label{eqn:goal} \P( |\Per(M_n)| \geq n^{(\frac{1}{2}-\eps_0)n} ) \geq
1-O(n^{-\Omega(1)}),\end{equation} as $n \rightarrow \infty$.

\vskip2mm

We shall do this by first establishing lower bounds on many minors
of $\Per(M_n)$, starting with $1 \times 1$ minors and increasing the
size of the minors one at a time, until reaching the full $n \times
n$ matrix $M_n$.  The main point will be to ensure that lower bounds
on $k \times k$ minors are passed on to many ``children'' $(k+1)
\times (k+1)$ minors, and that the lower bounds improve by almost
$n^{1/2}$ for the majority of  $k$.

\vskip2mm

When we talk about a $k \times k$ minor (of $M_n$), we always
understand that it is formed by some $k$ columns and  the
\emph{first} $k$ rows. Thus, such a
 minor can be indexed by its $k$
columns, which can be identified with an element of $\binom{[n]}{k}
:= \{ A \subset [n] := \{1,\ldots,n\}: |A| = k \}$.
 We use $M_A$ to denote the minor of $M_n$ associated
 to such an element $A \in \binom{[n]}{k}$.
 We also use $M_k$ to denote the $k \times n$ matrix formed by the first $k$ rows of $M_n$,
 thus $M_A$ is completely determined by $M_k$.

Let $1 \leq k \leq n$.  For any $A \in \binom{[n]}{k}$ and $\lambda
> 0$, we say that $A$ is \emph{$\lambda$-heavy} if
$|\Per(M_A)| \geq \lambda$.
 For any $N > 0$, let $E_{k,N,\lambda}$ denote the event that at least $N$
 elements of $ \binom{[n]}{k}$ are $\lambda$-heavy.  For instance, it is clear that
\begin{equation}\label{en1}
 \P( E_{1, n, 1} ) = 1.
\end{equation}
Our objective is to show that
\begin{equation}\label{en2}
 \P( E_{n, 1, n^{(\frac{1}{2}-\eps)n}} ) \geq 1-O(n^{-\Omega(1)}).
\end{equation}

Our strategy will be to move from the $k=1$ bound \eqref{en1} to the
$k=n$ bound \eqref{en2} by ``growing'' $N$ and $\lambda$ for many
values of $k$.

\smallskip

For small values of $k$ (e.g. $k \leq \eps n$, for some small $\eps$ to be chosen later) we will
just use a crude bound that does not grow $N$ or
$\lambda$, but has an exponentially high probability of success:

\begin{proposition}[Maintaining a single large minor]\label{early}  Let $1 \leq k < n$ and $\lambda > 0$.  Then we have
$$ \P( E_{k+1,1,\lambda} | E_{k, 1, \lambda} ) \geq 1 - 2^{-(n-k)}.$$
\end{proposition}

\noindent This result is quite easy and is established in Section
\ref{early-sec}.

\smallskip

Proposition \ref{early} does not grow $N$ or $\lambda$.  To handle the intermediate values of
 $k$ (e.g. between $\eps n$ and $(1-\eps) n$) we will need more sophisticated estimates.
  We first need a variant of Proposition \ref{early} in which the number $N$ of minors can be large.

\begin{proposition}[Maintaining many large minors]\label{maintain} Let
$1 \leq k \leq (1-\eps)n$ for some $\eps > 0$, let $N \geq 1$ and
let $\lambda > 0$.  Then we have
$$ \P( E_{k+1,\eps N/6,\lambda} | E_{k, N, \lambda} ) \geq 1 -  \exp( - \Omega(\eps n)  ).$$
\end{proposition}

We prove Proposition \ref{maintain} in Section \ref{early-sec}.  This proposition has a very small failure rate, but does not improve either $N$ or $\lambda$.  To achieve such growth, we need a further proposition, which has much higher failure rate but has a good chance of increasing either $N$ or $\lambda$ significantly.

\begin{proposition}[Growing many large minors]\label{grow} Let $1 \leq k \leq (1-\eps)n$ for some
$\eps > 0$, let $N \geq 1$, let $1> c > 0$, and let $\lambda > 0$.
Then we can partition the event $E_{k,N,\lambda}$ as
$E'_{k,N,\lambda,c} \vee E''_{k,N,\lambda,c}$, where the events
$E'_{k,N,\lambda,c}, E''_{k,N,\lambda,c}$ depend only on $M_k$, and
where
\begin{equation}\label{egrow}
 \P( E_{k+1,n^c N,\lambda} | E'_{k, N, \lambda,c} ) \geq 1/3
\end{equation}
and
\begin{equation}\label{cgrow}
 \P( E_{k+1,\eps N/4,n^{1/2-c} \lambda} | E''_{k, N, \lambda,c} ) \geq 1 - n^{-c/4}.
\end{equation}
\end{proposition}

\noindent This proposition will be proven in Section \ref{grow-sec}.
 Finally, to handle the last few values of $k$ ($(1-\eps) n \leq k \leq n$)
 we need the following result.

\begin{proposition}[Endgame]\label{endgame}  Let $1 \leq k \leq (1-\eps)n$ for some $\eps > 0$, and let $\lambda > 0$.  Then
$$ \P( E_{n,1,n^{-\log n} \lambda} | E_{k,1,\lambda} ) \geq 1 -  n^{-\Omega(1)} $$
if $n$ is sufficiently large depending on $\eps$.
\end{proposition}

\noindent This proposition will be proven in Section
\ref{endgame-sec}.

\vskip2mm

In the rest of this section, we show how Propositions
\ref{early}-\ref{endgame} imply the desired bound \eqref{en2}.

Recall that $\eps_0 > 0$ is fixed.  We choose a number $\eps > 0$ sufficiently small
compared to $\eps_0$, and a number $\eps'$  sufficiently small
compared to $\eps$. Let $k_1 := \lfloor (1-\eps) n \rfloor$.  In
view of Proposition \ref{endgame}, it suffices to show that
\begin{equation}\label{ek1}
\P( E_{k_1, 1, n^{(1/2-\eps_0/2})n} ) \geq 1 - n^{-\Omega(1)} .
\end{equation}

Applying Proposition \eqref{early} repeatedly, combined with
\eqref{en1}, we  obtain
\begin{equation}\label{kn}
 \P( E_{k_0,1,1} ) \geq 1 - \exp( - \Omega( n ) )
\end{equation}
for $k_0 := \lfloor \eps n \rfloor + 1$. (One can also use here
Alon's observation from the introduction, replacing $k_0$ with $2^m-1$ for some suitable $m$. However, this observation is specific to the permanent (as opposed to the determinant).)

\vskip2mm

To get from $k_0$ to $k_1$, we construct random variables
$N_k,\lambda_k$ and $W_k$  for $k_0 \leq k \leq k_1$ by the
following algorithm.

\begin{itemize}
\item Step 0.  Initialise $k:=k_0$. If $E_{k_0, 1,1}$ holds, then set $N_{k_0} := 1,\lambda_{k_0} :=
1, W_{k_0}=0$. Otherwise, set $N_{k_0}:=0, \lambda_{k_0}:=1,
W_{k_0}:=0$.

\item Step 1.  If  $N_k=0$ then set $N_{k+1}:=0$,
$\lambda_{k+1}:=\lambda_k$, $W_{k+1}:=W_k$. Move to Step 5.
  Otherwise,  move on to Step 2.

\item Step 2.  If $k=k_1$ then terminate the algorithm.  Otherwise, move on to Step 3.

\item Step 3.  By Proposition \ref{grow}, we are either in event $E'_{k,N_k,\lambda_k,\eps}$
or $E''_{k,N_k,\lambda_k,\eps}$.  Expose the row $k+1$.

\item Step 4.  Define $N_{k+1}$ and $\lambda_{k+1}$ by the following rule:
\begin{itemize}

\item[(I)] If $E'_{k,N_k,\lambda,\eps} \wedge  E_{k+1,n^{\eps} N_k / 4,\lambda_k}$ holds then we say
that $k$ is \emph{Type I}.  \newline Set $N_{k+1} := n^{\eps} N_k/4$
and $\lambda_{k+1} := \lambda_k$.

\item[(II)]  If $E'_{k,N_k,\lambda,\eps} \wedge \overline{E_{k+1,n^{\eps} N_k / 4,\lambda_k}} \wedge
E_{k+1,\eps N_k/6,\lambda_k}$ holds then we say that $k$ is
\emph{Type II}.  Set $N_{k+1} := \eps' N_k$ and $\lambda_{k+1} :=
\lambda_k$. (Here we use the fact that $\eps' \le \eps/6$.)

\item[(III)] If $E''_{k,N_k,\lambda,\eps} \wedge E_{k+1,\eps'
N_k,n^{1/2-\eps} \lambda_k}$ holds then we say that $k$ is
\emph{Type III}.  Set $N_{k+1} := \eps' N_k, \lambda_{k+1} :=
n^{1/2-\eps} \lambda_k$.

\item[(IV)] If $E''_{k,N_k,\lambda,\eps} \wedge \overline{E_{k+1,\eps' N_k,n^{1/2-\eps} \lambda_k}}
\wedge E_{k+1,\eps N_k/6,\lambda_k}$ holds then we say that $k$ is
\emph{Type IV}.   Set $N_{k+1} := \eps' N_k$ and $\lambda_{k+1} :=
\lambda_k$. (Here we use the fact that $\eps' \le \eps/6$.)

\item[(V)] If none of the above holds then set $N_{k+1}:=0,
\lambda_{k+1}:=\lambda_k$.
\end{itemize}

\vskip2mm  \noindent Set $W_{k+1}:= W_k + (1-\eps/2) - 3\I_{k
\,\,\hbox{type}\,\, I} -\I_{k \,\,\hbox{type} \,\,III}$.

 \item Step 5. Increment $k$ to $k+1$, and then return to Step 1.
\end{itemize}

\vskip2mm

We say that the algorithm is {\it successful} if at the terminating
time ($k=k_1$), $N_{k_1} \neq 0$ and $W_{k_1} \le \eps' n/2$. We
first show

\begin{proposition} The probability that the algorithm is successful is $1-\exp( - \Omega (\eps' n ) )$.
\end{proposition}

\begin{proof}  From \eqref{kn}, we know that the probability of failure at $k=k_0$ is
 $ \exp( - \Omega(n) )$. From Proposition \ref{maintain}, we know
that the probability that $E_{k+1,\eps N_k/6, \lambda_k}$ fails
given $E_{k, N_k, \lambda_k}$,  for any given $k_0 < k \leq k_1$, is
$\exp(-\Omega (\eps n)) \le \exp( - \Omega(\eps' n ))$.  The union
bound then implies that the probability that $N_{k_1}=0$ is
$\exp(-\Omega (\eps' n))$.

From Proposition \ref{grow} and the definition of $W_{k+1}$, we
obtain the submartingale-type property
$$ \E( W_{k+1}| M_k ) \leq W_k.$$
Also we have $|W_{k+1}-W_k| =O(1)$.  By Lemma \ref{azuma} (with the
$\sigma$-  algebra generated by $M_k$ playing the role of $\F_k$),
we obtain
$$ \P( W_{k_1} \geq \eps' n/2 ) \le \exp( - \Omega( \eps' n ) ).$$
The claim follows.\end{proof}

Next, we prove the following (deterministic) proposition,  which,
together with the previous proposition, imply \eqref{ek1}.

\begin{proposition} If the algorithm is successful, then $E_{k_1,1,
n^{(1/2-\eps_0/2)n}}$ holds.
\end{proposition}

\begin{proof} Assume that the algorithm is successful. We have
$W_{k_1} \le \eps' n/2$, which implies (via the definition of
$W_{k+1}$) that

$$\sum_{i=k_0}^{k_1-1} 3\I_{k \,\,\hbox{type}\,\, I }+ \I_{k \,\,\hbox{type
}\,\,III} \ge (k_1-k_0) - \eps'n . $$

On the other hand, the number of steps of type I is only $o(n)$.
Indeed, each such step increases $N_k$ by a huge factor $n^{\eps}/4$
while any other step decreases $N_k$ by at most a constant factor.
These combined with the fact that $N_k \le 2^n$ for any $k$ yield
the desired bound. Thus, the number of steps of type III is at least

$$(k_1-k_0) -(\eps' + o(1)) n \ge (1- 2\eps-O(1)- (\eps'+o(1)))n \ge (1-3\eps) n$$
thanks to the definition of $k_0, k_1$ and the fact that $\eps$ is
larger than $\eps'$. Since each type III step increases $\lambda_k$
by  $n^{1/2-\eps}$, it follows that

$$\lambda_{k_1} \ge n^{(1/2-\eps)(1-3 \eps)} \ge n^{(1-\eps_0/2)n }
$$
as we set $\eps$ much smaller than $\eps_0$. The proof is
complete. \end{proof}

\begin{remark} The above consideration in fact gives an
exponentially small
 probability bound for \eqref{ek1}. Unfortunately, the argument
used to prove Proposition \ref{endgame} only yields a polynomial
bound, especially in the last step of the argument (dealing with the
bottom row of $M_n$). This  is why the final bound in Theorem
\ref{main-thm} is only polynomial in nature.
\end{remark}

It remains to prove Propositions \ref{early}-\ref{endgame}.  This will be the focus of the remaining sections.

\section{Child and parent minors}\label{early-sec}

To prove Propositions \ref{early}-\ref{endgame}, it is important
to understand the relationship between the permanent of a ``parent''
minor $M_A$ and the permanent of a ``child'' minor $M_{A'}$.
More precisely, we say that $M_{A'}$ is a \emph{child} of $M_A$
(or $M_A$ is a \emph{parent} of $M_{A'}$) if we have $A' = A \cup \{i\}$
for some $i \not \in A$ (or equivalently if $A = A' \backslash \{i\}$ for some $i \in A'$).

Let  $A \in \binom{[n]}{k+1}$ for some $1 \leq k < n$. From the
definition of permanent we have the cofactor expansion

\begin{equation}\label{random}
 \Per(M_A) = \sum_{i \in A} a_{k+1, i} \Per(M_{A \backslash \{i\}}).
\end{equation}
We can draw an easy consequence of this:

\begin{lemma}[Large parent often has large child]\label{lplc}  Let $A \in \binom{[n]}{k}$
for some $1 \leq k < n$, and let $i \not \in A$. Assume that the
submatrix $M_k$ is fixed and we expose the (random) row $k+1$. Then

$$ \P\left( |\Per(M_{A \cup \{i\}})| \geq |\Per(M_A)| \right) \geq \frac{1}{2}.$$
In fact, this bound is still true if we condition on all the entries of the row $k+1$ except for $a_{k+1,i}$.
\end{lemma}

\begin{proof} Let $M'_{A \cup
\{i\}}$ denote the same minor as $M_{A \cup \{i\}}$ but with the
sign $a_{k+1,i} \in \{-1,+1\}$ replaced by $-a_{k+1,i}$. From
 \eqref{random}, we have
$$ |\Per(M_{A' \cup \{i\}}) - \Per(M_{A \cup \{i\}})| = 2 |\Per(M_A)|.$$
The claim follows.
\end{proof}

We can amplify this  probability $\frac{1}{2}$ to an exponentially
small probability by exploiting the fact that one parent has many
``independent'' children.

\begin{lemma}[Large parent often has many large children]\label{lplc2}
Let $A \in \binom{[n]}{k}$ for some $1 \leq k < n$, and let $I
\subset [n] \backslash A$. Assume that the submatrix $M_k$ is fixed
and we expose the (random) row $k+1$.   Then
\begin{equation}\label{claim1}
\P( |\Per(M_{A \cup \{i\}})| \geq |\Per(M_A)| \hbox{ for some } i \in I ) \geq 1 - 2^{-|I|}
\end{equation}
and
\begin{equation}\label{claim2}
 \P( |\Per(M_{A \cup \{i\}})| \geq |\Per(M_A)| \hbox{ for at least } |I|/3 \hbox{ values of } i \in I )
 \geq 1 - O( \exp( -\Omega( |I| )) )
\end{equation}
\end{lemma}

\begin{proof} We further condition on all
entries of the $k+1$ row except for $a_{k+1,i}$ where $i \in I$. The
first claim follows from the previous lemma and independence.
 The second follows from Chernoff's bound. (One can, of course, use
 Azuma's inequality as well.)
\end{proof}

We can now immediately prove Proposition \ref{early}:

\begin{proof}[Proof of Proposition \ref{early}]  Let us condition on
the first $k$ rows $M_k$, and assume that $E_{k,1,\lambda}$ holds,
thus there exists a $\lambda$-heavy $A \in \binom{[n]}{k}$.
Applying \eqref{claim1} with $I := [n] \backslash A$ we conclude that
$$\P( A' \hbox{ is } \lambda-\hbox{heavy for some } A' \in \binom{[n]}{k+1} ) \geq 1 - 2^{-(n-k)}$$
and the claim follows.
\end{proof}

A slightly more sophisticated argument also gives Proposition
\ref{maintain}.

\begin{proof}[Proof of Proposition \ref{maintain}] We may take $N$ to be an integer.
Let us condition on the first $k$ rows $M_k$, and assume that
$E_{k,N,\lambda}$ holds, thus there exist $N$ $\lambda$-heavy minors
$A_1,\ldots,A_N \in \binom{[n]}{k}$. Each $A_j$ has at least $\eps
n$ children $A_j \cup \{i\}$.  Let us call $A_j$ \emph{good} if it
has at least $\eps n/3$ $\lambda$-heavy children $A_j \cup \{i\}$.
By \eqref{claim2}, each $j$ has a probability $1 -  \exp( -\Omega(
\eps n ))$ of being good.  Applying Lemma \ref{first} with $c :=
1/2$, we conclude that with probability $1 -  \exp( -\Omega( \eps n
))$, at least $N/2$ of the $j$ are good.

Let us now suppose that at least $N/2$ of the $j$ are good.
By definition, each good $A_j$ has at least $\eps n/3$ $\lambda$-heavy children $A_j \cup \{i\}$. On the other hand, each child has at most $n$ parents.  By the usual double counting argument, this implies that at least $\eps N/6$ elements in $A' \in \binom{[n]}{k+1}$ are $\lambda$-heavy, and the claim follows.
\end{proof}

\section{Growing large minors}\label{grow-sec}

The purpose of this section is to prove Proposition \ref{grow}.  Fix
$k,\eps,N,c,\lambda$; we may take $N$ to be an integer.  We
condition on  $M_k$ of $M_n$ and assume that $E_{k,N,\lambda}$
holds. Thus we may find distinct $\lambda$-heavy $A_1,\ldots,A_N \in
\binom{[n]}{k}$.

For each $l \geq 1$, let $F_l$ denote the number of $A' \in \binom{[n]}{k+1}$ which have exactly $l$ parents in the set $\{A_1,\ldots,A_N\}$.  Since each $A_j$ has at least $\eps n$ children $A_j \cup \{i\}$, a double counting argument shows
$$ \sum_{l=1}^n l F_l \geq \eps n N.$$
Now set $K := \lfloor \frac{\eps}{8} n^{1-c} \rfloor$.  Since
$$ \sum_{l=1}^n l F_l \leq K ( F_1 + \ldots + F_K ) + n (F_{K+1} + \ldots + F_n)$$
we see that either
\begin{equation}\label{kfk}
F_1 + \ldots + F_K \geq \frac{\eps n N}{2K}
\end{equation}
or
\begin{equation}\label{kfk2}
F_{K+1} + \ldots + F_n \geq \frac{\eps N}{2}.
\end{equation}

We let $E'_{k,N,\lambda,c}$ be the event that \eqref{kfk} (and $E_{k,N,\lambda}$, of course) holds, and $E''_{k,N,\lambda,c}$ be the event that \eqref{kfk} fails but \eqref{kfk2} (and $E_{k,N,\lambda}$) holds.

Suppose first that $E'_{k,N,\lambda,c}$ holds.  Then by \eqref{kfk},
we can find at least $\frac{\eps n N}{2K}$ elements $A'$ in
$\binom{[n]}{k+1}$, each of which has at least one parent in
$\{A_1,\ldots,A_N\}$.  By Lemma \ref{lplc}, each such $A'$ is
$\lambda$-heavy with probability at least $1/2$.  Applying Lemma
\ref{first}, we conclude that with probability at least $1/3$, at
least $\frac{\eps n N}{8K}$ of these $A'$ will be $\lambda$-heavy.
The claim \eqref{egrow} now follows from the choice of $K$.

Now suppose instead that $E''_{k,N,\lambda,c}$ holds.  Then by
\eqref{kfk2}, we can find at least $\frac{\eps N}{2}$ elements $A'$
in $\binom{[n]}{k+1}$, each one of which has at least $K$ parents in
$\{A_1,\ldots,A_N\}$.  By \eqref{random} and Lemma \ref{lo}, we see
that each of these $A'$ is $n^{1/2-c} \lambda$-heavy with
probability $1 - O( n^{1/2-c} / K^{1/2} )$.  Applying Lemma
\ref{first}, we see that with probability $1 - O( n^{1/2-c} /
K^{1/2} )$, at least $\eps N/4$ of the $A'$ will be $n^{1/2-c}
\lambda$-heavy.  The claim \eqref{cgrow} now follows from the choice
of $K$ (and the assumption that $n$ is large).  This concludes the
proof of Proposition \ref{grow}.

\section{The endgame}\label{endgame-sec}

The purpose of this section is to prove Proposition \ref{endgame}.  Fix $k, n,\lambda$.  We
 condition on  $M_k$ and assume that $E_{k,1,\lambda}$ holds,
 thus one of the elements of $ \binom{[n]}{k}$ is $\lambda$-heavy.  By symmetry we may assume without
 loss of generality that $[k]$ is $\lambda$-heavy.  Our task is to show
 that $[n]$ is $n^{-\log n}\lambda$-heavy with probability $1-n^{-\Omega(1)}$.

Set $L := \frac{1}{100} \lfloor \log n \rfloor$.  We first show that
there are plenty of heavy minors in $\binom{[n]}{n-L}$.

\begin{lemma}[Many heavy minors of order $n-L$]\label{heavy-nl}
Let $B \subset \binom{[n] \backslash [k]}{2L}$.  Then with probability
 $1 - \exp(-\Omega(L))$, there exists a $\lambda$-heavy minor $A \in \binom{[n]}{n-L}$ which contains $[n] \backslash B$.
\end{lemma}

\begin{proof}  We construct $A_j \in \binom{[n]}{j}$ for $k \leq j \leq n-L$ by the following algorithm.
\begin{itemize}
\item Step 0.  Initialise $j := k$ and $A_j := [k]$.
\item Step 1.  If there exists $i \in [n] \backslash (B \cup A_j)$ such that $A_j \cup \{i\}$ is $\lambda$-heavy, then choose one of these $i$ arbitrarily, set $A_{j+1} := A_j \cup \{i\}$, and go onto Step 4.  Otherwise, go to Step 2.
\item Step 2.  If there exists $i \in B \backslash A_j$ such that $A_j \cup \{i\}$ is $\lambda$-heavy, then choose one of these $i$ arbitrarily, $A_{j+1} := A_j \cup \{i\}$, and go onto Step 4.  Otherwise, go to Step 3.
\item Step 3.  Choose $i \in [n] \backslash A_j$ arbitrarily, and set $A_{j+1} := A_j \cup \{i\}$.
\item Step 4.  If $j = n-L-1$ then {\tt STOP}.  Otherwise increment $j$ to $j+1$ and return to Step 1.
\end{itemize}

Applying \eqref{claim1} we see that if $A_j$ is $\lambda$-heavy for
some $k \leq j < n-L$, then with probability at least $1 -
2^{-(n-j)}$ $A_j \cup \{i\}$ is $\lambda$-heavy for at least one $i
\in [n] \backslash A_j$. By construction, this implies that
$A_{j+1}$ is $\lambda$-heavy with probability at least $1 -
2^{-(n-j)}$.  By the union bound (and the fact that $A_k$ is
$\lambda$-heavy),
 we thus conclude that with probability $1 - O(2^{-L})$, $A_j$ is $\lambda$-heavy for all $k \leq j \leq n-L$.

Let $W_j := |[n] \backslash (B \cup A_j)|$, thus $W_k = n-k-2L$ and
$\min(W_j-1,0) \leq W_{j+1} \leq W_j$ for all $k \leq j < n-L$.  By
\eqref{claim1}, we see that if $A_j$ is $\lambda$-heavy, and $W_j >
0$ then $W_{j+1} = W_j - 1$ with probability at least $1 -
2^{-W_j}$.  By the union bound, we conclude that $W_{n - \lfloor
2.01 L \rfloor} = \lfloor 2.01 L \rfloor - 2L$ with probability
$1-\exp(-\Omega(L))$. We condition on this event.

For any $n-\lfloor 2.01 L \rfloor \leq j < n-L$, we see from the
previous discussion that if $W_j > 0$, then $W_{j+1} = W_j-1$ with
probability at least $0.4$ (say), and $W_{j-1}=W_j$ otherwise.  From
this we see that
$$ \E( 2^{W_{j+1}} - 1 | W_{n - \lfloor 2.01 L \rfloor} = \lfloor 2.01 L \rfloor - 2L )
\leq \frac{1}{\sqrt{2}} \E( 2^{W_j} - 1 | W_{n - \lfloor 2.01 L \rfloor} = \lfloor 2.01 L \rfloor - 2L )$$
(say) for all $n-\lfloor 2.1 L \rfloor \leq j < n-L$.  Since
$$ \E( 2^{W_{n - \lfloor 2.01 L \rfloor}} - 1 | W_{n - \lfloor 2.01 L \rfloor} =
\lfloor 2.01 L \rfloor - 2L ) \le 2^{0.01 L},$$ we conclude by
iteration that
$$ \E( 2^{W_{n-L}} - 1 | W_{n - \lfloor 2.01 L \rfloor} = \lfloor 2.01 L \rfloor - 2L ) \le \exp( - \Omega(L) )$$
and thus $W_{n-L}=0$ with probability $1-\exp(-\Omega(L))$.  Since
$A_{n-L}$ is also $\lambda$-heavy with probability
$1-\exp(-\Omega(L))$, the claim follows.
\end{proof}

For any integer $N \geq 1$, any $1 \leq j \leq L$, and any $\lambda'
> 0$, let $F_{j,N,\lambda'}$ denote the event that there exists $N$
$\lambda'$-heavy sets (minors) $A_1,\ldots,A_N \in \binom{[n]}{n-j}$
whose complements $[n] \backslash A_1, \ldots, [n] \backslash A_N$
are disjoint.

\begin{corollary}[Many complement-disjoint heavy minors of order $n-L$]\label{coro}
We have $\P( F_{L, \lfloor\eps n/10L\rfloor,\lambda} ) =
1-\exp(-\Omega(L)) = 1 - n^{-\Omega(1)}$.
\end{corollary}

\begin{proof} Choose $\lfloor \eps n/4L\rfloor$ disjoint sets
$B_1,\ldots,B_{\lfloor \eps n/4L\rfloor} \in \binom{[n]\backslash [k]}{2L}$
 arbitrarily.  For each of these $B_i$, Lemma \ref{heavy-nl}
 shows that with probability $1 - \exp(-\Omega(L))$, there
 exists a heavy $A_i \in \binom{[n]}{n-L}$ with
 $[n] \backslash A_i \subset B_i$ (in particular, the sets $[n] \backslash A_i$ are disjoint).
  The claim now follows from Lemma \ref{first}.
\end{proof}

We now propagate the events $F_{j,N}$ downward from $j=L$ to $j=1$ (accepting some loss in the weight threshold $\lambda'$ and in the population $N$ of heavy minors when doing so) by means of the following lemma.

\begin{lemma}[Many heavy minors of order $n-j$ imply many heavy minors of order $n-j+1$]\label{lemur} Let $1 < j \leq L$, $N \geq n^{0.5}$ (say), and $\lambda' > 0$.  Then
\begin{equation}\label{fjn}
 \P( F_{j-1, \lfloor N/10 \rfloor, \lambda'/n} | F_{j,N,\lambda} ) \geq 1 -  n^{-\Omega(1)} .
\end{equation}
\end{lemma}

\begin{proof} Fix $j, N$.  We condition on $M_{n-j}$ so that
$F_{j,N}$ hold.  Thus we can find $\lambda'$-heavy sets
$A_1,\ldots,A_N \in \binom{[n]}{n-j}$ with disjoint complements,
which we now fix. For each $A_i$, we arbitrarily choose a child $B_i
= A_{i }\cup \{h_{i} \} \in \binom{[n]}{n-j+1}$. By construction,
the $B_1,\ldots,B_N$ also have disjoint complements and the $h_i$
are different.


Let $T := \lfloor n^{0.1} \rfloor$.  Call a child $B_i$ \emph{good}
if it has at least $T$ $\lambda'/n$-heavy parents (of which $A_i$
will be one of them), and \emph{bad} otherwise.  There are two
cases.

\vskip2mm

{\it Case 1: at least half of the $B_i$ are good.} By \eqref{random}
and Lemma \ref{lo}, each $B_i$ has a probability $1 - O(T^{-1/2})$
of being $\lambda'/n$-heavy.  The claim now follows from Lemma
\ref{first}.

\vskip2mm

{\it Case 2: at least half of the $B_i$ are bad.} Let $I$ be the set
of all $i$ such that $B_{i}$ is bad and $H$ be the set of $h_{i}$,
$i \in I$. Draw a bipartite graph $G$ between $I$ and $H$ by
connection $i$ to $h_{{i'}}$ if $B_{i } \backslash \{h_{i'}\}$ is
$\lambda'/n$-heavy. As the $B_{i}$ are bad, each $i \in I$ has
degree at most $T$. By double counting the edges in this graph, we
have

$$\sum_{h \in H} \deg_{h} \le T|I| =T|H|  $$ where $\deg_h$ denotes
the degree of $h$.

Again by a double counting argument, one can easily shows that the
set $I':= \{i| \deg_{h_{i}} \le 2T \}$ is at least $|I|/2 \ge N/4$.
We condition on the entries of the $n-j+1$ row not in the columns
determined by $I'$. For each $i\in I'$, let

$$Y_{i}:= \min \left(\frac{|\Per M_{B_{i}}| }{\lambda'} ,1 \right) $$

\noindent and $Y:=\sum_{i \in I'} Y_{i}$. By Lemma \ref{lplc},
$\E(Y_{i}) \ge 1/2$ since each $B_{i}$ has a $\lambda'$-heavy
parent. Thus, by linearity of expectation, $\E(Y) \ge |I'|/2 \ge
|N/8|$.

Now we estimate the effect of each random entry  $a_{n-j+1,h}$ on
$Y$. If $h \notin B_{i}$, then flipping $a_{n-j+1,h}$ does not
change $Y_{i}$.  If $h \in B_{i}$ and the $(n-j)\times (n-j)$ minor
corresponding to $a_{n-j+1,h}$ is not $\lambda'/n$-heavy, then
flipping $a_{n-j+1,h}$ changes $Y_{i}$ by at most $2/n$. Finally, if
$h \in B_{i}$ and the $(n-j)\times (n-j)$ minor corresponding to
$a_{n-j+1,h}$ is  $\lambda'/n$-heavy, then flipping $a_{n-j+1,h}$
changes $Y_{i}$ by at most $1$. On the other hand, the number of
such $i$ is at most $2T$ by the definition of $I'$. Thus,  flipping
$a_{n-j+1,h}$ changes $Y$ by at most $2T+2 \le 3T$.

\vskip2mm

By Lemma \ref{Azuma0} and the definitions of $N$ and $T$

$$\P(|Y- \E(Y)| \ge |I'|/100) \le 2 \exp\left(-\Omega (\frac{|I'|^{2}}{T^{2} |I'|})\right) = \exp\left(-\Omega (\frac{N}{T^{2}})\right) =
n^{-\Omega (1)}. $$

Since $\E(Y) \ge |I'|/2 \ge N/8$, it follows that $Y \ge N/9$ with
probability $1-n^{-\Omega(1)}$. Finally, notice that if $Y \ge N/9$,
then the definition of $Y$ and  $Y_i$ implies (with room to spare)
that for at least $N/10$ indices $i$, $\frac{|\Per M_{B_{i}}|
}{\lambda'} \ge \frac{1}{n}$. This concludes the proof.
\end{proof}

Iterating Lemma \ref{lemur} $L \le \frac{\log n}{100} $ times
starting with Corollary \ref{coro}, we conclude that
$$
 \P( F_{1, \lfloor n^{0.5} \rfloor, n^{-\log n} \lambda} ) \geq 1 -  n^{-\Omega(1)} .
$$
Now suppose that $F_{1, \lfloor n^{0.5} \rfloor, n^{-\log n}
\lambda}$ holds, thus there are at least $\lfloor n^{0.5}\rfloor$
$n^{-\log n} \lambda$-heavy minors in $\binom{[n]}{n-1}$.  Applying
\eqref{random} and Lemma \ref{lo} we conclude that $[n]$ is
$n^{-\log n} \lambda$-heavy with probability at least $1 - O(
1/\sqrt{\lfloor n^{0.5}\rfloor}) = 1 - n^{-\Omega(1)}$.  This
completes the  proof of Theorem \ref{main-thm}.


\begin{thebibliography}{10}

\bibitem{AS} N. Alon, J. Spencer, The probabilistic method, 2nd
Edition, Wiley, 2000.



\bibitem{BVW}
J. Bourgain, V. Vu, P. Wood, \emph{On the singularity probability of
random discrete matrices}, preprint.

\bibitem{Erd1}
P. Erd\H{o}s, \emph{On a lemma of Littlewood and Offord}, Bull. Amer.
Math. Soc. \textbf{51} (1945), 898--902.

\bibitem{janson}
S. Janson, \emph{The number of spanning trees, Hamilton cycles, and perfect matchings in a random graph}, Combin. Probab. Comput. \textbf{3} (1) (1994), 97--126.

\bibitem{KKS}
J. Kahn, J. Koml\'os, E. Szemer\'edi, \emph{On the probability that a
random $\pm 1$ matrix is singular}, J. Amer. Math. Soc.
\textbf{8} (1995), 223--240.

\bibitem{Kom1} J. Koml\'os, \emph{On the determinant of $(0,1)$ matrices},
 Studia Sci. Math. Hungar. \textbf{2} (1967) 7-22.

\bibitem{Kom2} J. Koml\'os,
\emph{On the determinant of random matrices},
 Studia Sci. Math. Hungar. \textbf{ 3}
(1968) 387--399.

\bibitem {Kom3} J. Koml\'os, Circulated note, (reproduced as Theorem
14.11 in Bollob\'as book ``Random graphs'', Cambridge Univ. Press,
Second Edition, 2001).

\bibitem{rempala}
G. Rempala, A. Gupta, \emph{Some extensions of Girko's limit theorems for permanents of random matrices}, Random Oper. Stochastic Equations \textbf{8} (2000), no. 4, 305--318.

\bibitem{rempala2}
G. Rempala, J. Weso\l owski, \emph{Asymptotic behavior of random
permanents}, Statist. Probab. Lett. \textbf{45} (1999), 149--158.

\bibitem{TV1} T. Tao and V. Vu, \emph{On random $\pm 1$ matrices: Singularity
and  Determinant},  Random Structures Algorithms \textbf{28} (2006),
no. 1,
 1--23.

\bibitem{TV2} T. Tao and V. Vu, \emph{On the
singularity probability of random Bernoulli matrices}, Journal of
the A. M. S. 20 {\bf 3}, 2007, 603-628.

\bibitem{Yu} Y. Yu, \emph{The permanent rank of a matrix}, J. Combin. Thy. \textbf{85} (1999), 237--242.

\bibitem{Wan} I. M. Wanless, \emph{Permanents of matrices of signed ones}, Linear and
Multilinear Algebra, \textbf{53} (2005) 427-433.



\end{thebibliography}
\end{document}